\documentclass[10pt]{article}
\usepackage{amsthm,amsmath,amsfonts,latexsym,amssymb,mathrsfs}

\theoremstyle{definition}
\newtheorem{Def}{Definition}[section]

\newtheorem{thm}[Def]{Theorem}
\newtheorem{cor}[Def]{Corollary}
\newtheorem{lem}[Def]{Lemma}

\newtheorem{rem}[Def]{Remark}

\newtheorem{con}[Def]{Conjecture}

\renewcommand{\thefootnote}{\fnsymbol{footnote}}

\begin{document}

\title{\bf Which subnormal Toeplitz operators \\ are either normal or analytic\,?}
\author{\Large Ra{\'u}l\ E.\ Curto, In Sung Hwang, and Woo Young Lee}

\date{}
\maketitle

\renewcommand{\thefootnote}{}
\footnote{
The work of the first named author was partially supported by NSF Grant
DMS-0801168. \ The work of the second named author was supported by NRF grant funded by
the Korea government(MEST)(2009-0075890). \ The work of the third
author was supported by the Basic Science Research Program through the
NRF grant funded by the Korea government(MEST)(2010-0001983).
}

\maketitle

\bigskip

\begin{abstract}
We study subnormal Toeplitz operators on the vector-valued Hardy space of the unit circle,
along with an appropriate reformulation of P.R. Halmos's Problem 5: Which subnormal
block Toeplitz operators are either normal or analytic\,? \ We extend and prove Abrahamse's Theorem
to the case of matrix-valued symbols; that is, we show that every subnormal block Toeplitz operator
with bounded type symbol (i.e., a quotient of two bounded analytic functions), whose analytic and
co-analytic parts have the ``left coprime factorization,'' is normal or analytic.
\ We also prove that the left coprime factorization
condition is essential. \ Finally, we examine a well known conjecture, of whether every submormal
Toeplitz operator with finite rank self-commutator is normal or analytic. \
\end{abstract}

%%%%%%%%%%%%%%%%%%%%%%%%%%%%%%%%%%%%%%%%%%%%%%%%%%%%%%%%%%%%%%%%%%%%%%%%%%%%%%%%%%%%%%%%%%%%%%%%%
%
%             SECTION 1
%
%%%%%%%%%%%%%%%%%%%%%%%%%%%%%%%%%%%%%%%%%%%%%%%%%%%%%%%%%%%%%%%%%%%%%%%%%%%%%%%%%%%%%%%%%%%%%%%%%%
\pagestyle{plain}

\vskip 1.5cm

\section{Introduction}

\noindent
Toeplitz operators arise naturally in several fields of mathematics and in a variety of problems in
physics (in particular, in the field of quantum mechanics). \ On the other hand, the theory
of subnormal operators is an extensive and highly developed area, which has made important contributions
to a number of problems in functional analysis, operator theory, and mathematical physics. \ Thus, it
becomes of central significance to describe in detail subnormality for Toeplitz operators. \ This paper
focuses on subnormality for \textit{block} Toeplitz operators and more precisely, the case of block Toeplitz
operators with bounded type symbols. \ Our main result is an appropriate generalization of Abrahamse's Theorem
to the case of matrix-valued symbols; that is, we show that every subnormal block Toeplitz operator
with bounded type symbol (i.e., a quotient of two bounded analytic functions),
whose analytic and co-analytic parts have the ``left coprime factorization,'' is normal or analytic.

Naturally, this research is closely related to the study of subnormal operators with finite rank self-commutator,
a class that has been extensively researched by many authors. \ However, until now a complete description of that
class has proved elusive. \ Recently, D. Yakubovich \cite{Ya} has shown that if $S$ is a pure subnormal operator
with finite rank self-commutator and admits a normal extension with no nonzero eigenvectors, then $S$ is
unitarily equivalent to a block Toeplitz operator with analytic rational normal matrix symbol. \ A corollary of our
main result illustrates, in a certain sense, the case of subnormal Toeplitz operators with finite rank self-commutator.

To describe our results in more detail, we first need to review a
few essential facts about (block) Toeplitz operators, and for that
we will use \cite{Do1}, \cite{Do2}, \cite{GGK}, and \cite{Ni}. \ Let
$\mathcal{H}$ be a complex Hilbert space and let $\mathcal{B(H)}$ be
the algebra of bounded linear operators acting on $\mathcal{H}$. \
An operator $T\in\mathcal{B(H)}$ is said to be {\it hyponormal} if
its self-commutator $[T^*,T]:= T^*T-TT^*$ is positive
(semi-definite), and {\it subnormal} if there exists a normal
operator $N$ on some Hilbert space $\mathcal{K}\supseteq
\mathcal{H}$ such that $\mathcal H$ is invariant under $N$ and
$N\vert_{\mathcal{H}}=T$. \ Let $\mathbb{T} \equiv
\partial\,\mathbb{D}$ be the unit circle in the complex plane. \ Let
$L^2\equiv L^2({\mathbb T})$ be the set of all square-integrable
measurable functions on $\mathbb{T}$ and let $H^2\equiv H^2({\mathbb
T})$ be the corresponding Hardy space. \ Let $H^\infty\equiv
H^\infty(\mathbb T):=L^\infty (\mathbb T)\cap H^2 (\mathbb T)$, that
is, $H^\infty$ is the set of bounded analytic functions on $\mathbb
D$. \ Given $\phi\in L^\infty$, the Toeplitz operator $T_\phi$ and
the Hankel operator $H_\phi$ are defined by
$$
T_\phi g:=P(\phi g) \quad\hbox{and}\quad H_\phi(g):=JP^\perp(\phi g)
\qquad (g\in H^2),
$$
where $P$ and $P^\perp$ denote the orthogonal projections that map from $L^2$ onto $H^2$ and $(H^2)^\perp$,
respectively, and where $J$ denotes the unitary operator on $L^2$ defined by $J(f)(z)=\overline z f(\overline z)$. \

In the early 1960's, normal Toeplitz operators were characterized by a property of their symbols by A. Brown and
P.R. Halmos \cite{BH}. \ On the other hand, the exact nature of the relationship between the symbol
$\phi\in L^\infty$ and the hyponormality of $T_\phi$ was understood much later, in 1988, via Cowen's theorem \cite{Co3}.
\bigskip

\noindent
{\bf Cowen's theorem.} (\cite{Co3}, \cite{NT}) {\it
For each $\phi\in L^\infty$, let
$$
\mathcal{E}(\phi)\equiv
\{k\in H^\infty:\ ||k||_\infty\le 1\ \hbox{and}\
\phi-k\overline\phi\in H^\infty\}.
$$
Then $T_\phi$ is hyponormal if and only if
$\mathcal{E}(\phi)$ is nonempty.
}
\bigskip

The elegant and useful theorem of C. Cowen has been used in the works \cite{CuL1}, \cite{CuL2}, \cite{FL},
\cite{Gu1}, \cite{Gu2}, \cite{GS}, \cite{HKL1}, \cite{HKL2}, \cite{HL1}, \cite{HL2}, \cite{HL3}, \cite{Le},
\cite{NT} and \cite{Zhu}, which have been devoted to the study of hyponormality for Toeplitz operators on $H^2$. \
When one studies hyponormality (also, normality and subnormality) of the Toeplitz operator $T_\phi$ one may,
without loss of generality, assume that $\phi(0)=0$; this is because hyponormality is invariant under translation
 by scalars. \ We now recall that a function $\phi\in L^\infty$ is said to be of {\it bounded type} (or
in the Nevanlinna class) if there are analytic functions $\psi_1,\psi_2\in H^\infty (\mathbb D)$ such that
$$
\phi(z)=\frac{\psi_1(z)}{\psi_2(z)}\quad\hbox{for almost all}\ z\in \mathbb{T}.
$$
It is well known \cite[Lemma 3]{Ab} that if $\phi\notin H^\infty$ then
\begin{equation}\label{1.1}
\hbox{$\phi$ is of bounded type}\ \Longleftrightarrow\ \hbox{ker}\, H_\phi\ne \{0\}\,.
\end{equation}
If $\phi\in L^\infty$, we write
$$
\phi_+\equiv P \phi\in H^2\quad\text{and}\quad \phi_-\equiv
\overline{P^\perp \phi}\in zH^2.
$$
Assume now that both $\phi$ and $\overline\phi$ are of bounded type. \ Since $T_{\overline z}H_\psi=H_\psi T_z$
for all $\psi \in L^{\infty}$, it follows from Beurling's Theorem that $\text{ker}\, H_{\overline{\phi_-}}=\theta_0 H^2$
and $\text{ker}\, H_{\overline{\phi_+}}=\theta_+ H^2$ for some inner functions $\theta_0, \theta_+$. \ We thus have
$b:={\overline{\phi_-}}\theta_0 \in H^2$, and hence we can write
$$
\phi_-=\theta_0\overline{b} \text{~and similarly~} \phi_+=\theta_+\overline{a} \text{~for some~} a \in H^2.
$$
In particular, if $T_\phi$ is hyponormal and $\phi\notin H^\infty$,
and since
$$
[T_\phi^*, T_\phi]=H_{\overline\phi}^* H_{\overline\phi}-H_\phi^* H_\phi=
H_{\overline{\phi_+}}^* H_{\overline{\phi_+}}-H_{\overline{\phi_-}}^* H_{\overline{\phi_-}},
$$
it follows that
$||H_{\overline{\phi_+}} f||\ge ||H_{\overline{\phi_-}} f||$
for all $f\in H^2$, and hence
\begin{equation*}
%%\label{1.2}
\theta_+ H^2= \text{ker}\, H_{\overline{\phi_+}}\subseteq
\text{ker}\, H_{\overline{\phi_-}}=\theta_0 H^2,
\end{equation*}
which implies that $\theta_0$ divides $\theta_+$, i.e.,
$\theta_+=\theta_0\theta_1$ for some inner function $\theta_1$.
We write, for an inner function $\theta$,
$$
\mathcal H_{\theta}:=H^2\ominus \theta\,H^2.
$$
Note that if $f=\theta \overline a \in L^2$, then
$f\in H^2$ if and only if $a\in \mathcal H_{z\theta}$;
in particular, if $f(0)=0$ then $a\in \mathcal H_\theta$.
Thus, if $\phi=\overline{\phi_-}+\phi_+\in L^\infty$ is
such that $\phi$ and $\overline\phi$ are of bounded type
such that $\phi_+(0)=0$ and $T_\phi$ is hyponormal, then we
can write
\begin{equation*}
%%\label{1.3}
\phi_+=\theta_0\theta_1\bar a\quad\text{and}\quad \phi_-=\theta_0
\bar b, \qquad\text{where $a\in \mathcal{H}_{\theta_0\theta_1}$ and
$b\in\mathcal{H}_{\theta_0}$.}
\end{equation*}
By Kronecker's Lemma
\cite[p. 183]{Ni}, if $f\in H^\infty$ then $\overline f$ is a
rational function if and only if $\hbox{rank}\, H_{\overline f}<\infty$,
which implies that
\begin{equation}\label{1.4}
\hbox{$\overline{f}$ is rational}
\ \Longleftrightarrow\
f=\theta\overline b\ \ \hbox{with a finite Blaschke product $\theta$}.
\end{equation}
On the other hand, M. Abrahamse \cite[Lemma 6]{Ab} also showed that if
$T_\phi$ is hyponormal, if $\phi\notin H^\infty$, and if $\phi$ or
$\overline{\phi}$ is of bounded type then both $\phi$ and
$\overline{\phi}$ are of bounded type.

\bigskip

We now introduce the notion of block Toeplitz operators. \ For a Banach space $\mathcal X$,
let $L^2_{\mathcal X}\equiv L^2_{\mathcal X}(\mathbb T)$ be the Hilbert space of
$\mathcal X$-valued norm square-integrable measurable
functions on $\mathbb{T}$ and let
$H^2_{\mathcal X}\equiv H^2_{\mathcal X}(\mathbb T)$ be the corresponding Hardy space.\
We observe that $L^2_{\mathbb{C}^n}= L^2\otimes \mathbb{C}^n$ and
$H^2_{\mathbb{C}^n}= H^2\otimes \mathbb{C}^n$.
If $\Phi$ is a matrix-valued function
in $L^\infty_{M_n}\equiv L^\infty_{M_n}(\mathbb T)$
($=L^\infty\otimes M_n$) then
$T_\Phi: H^2_{\mathbb{C}^n}\to H^2_{\mathbb{C}^n}$ denotes
the block Toeplitz operator with symbol $\Phi$
defined by
$$
T_\Phi F:=P_n(\Phi F)\quad \hbox{for}\ F\in H^2_{\mathbb{C}^n},
$$
where $P_n$ is the orthogonal projection
of $L^2_{\mathbb{C}^n}$ onto $H^2_{\mathbb{C}^n}$.\
A block Hankel operator with symbol
$\Phi\in L^\infty_{M_n}$ is the operator
$H_\Phi: H^2_{\mathbb{C}^n}\to H^2_{\mathbb{C}^n}$ defined by
$$
H_\Phi F := J_n P_n^\perp (\Phi F)\quad \hbox{for}\ F\in H^2_{\mathbb{C}^n},
$$
where $J_n$ denotes the unitary operator from
$\bigl(H^2_{\mathbb{C}^n}\bigr)^\perp$ to $H^2_{\mathbb{C}^n}$ given
by $J_n(F)(z):=\overline{z} I_n F(\overline{z})$ for $F \in
H^2_{\mathbb{C}^n}$, and where $I_n$ is the $n\times n$ identity matrix. \
If
we set $H^2_{\mathbb{C}^n}:=H^2\oplus\cdots\oplus H^2$ then we see
that
$$
T_\Phi=\begin{bmatrix} T_{\phi_{11}}&\hdots&T_{\phi_{1n}}\\
&\vdots\\
T_{\phi_{n1}}&\hdots&T_{\phi_{nn}}\end{bmatrix} \quad
\hbox{and}\quad
H_\Phi=\begin{bmatrix} H_{\phi_{11}}&\hdots&H_{\phi_{1n}}\\
&\vdots\\
H_{\phi_{n1}}&\hdots&H_{\phi_{nn}}\end{bmatrix},
$$
where
$$
\Phi=\begin{bmatrix} \phi_{11}&\hdots&\phi_{1n}\\
&\vdots\\
\phi_{n1}&\hdots&\phi_{nn}\end{bmatrix}\in
L^\infty_{M_n}.
$$
For $\Phi\in L^\infty_{M_n}$, write
$$
\widetilde\Phi (z):=\Phi^*(\overline z).
$$
A matrix-valued function
$\Theta\in H^\infty_{M_{n\times m}}$ ($=H^\infty\otimes M_{n\times m}$)
is called {\it inner} if
$\Theta(z)^*\Theta(z)=I_m$ for almost all $z\in\mathbb{T}$. \
The following basic
relations can be easily derived:
\begin{align}
&T_\Phi^*=T_{\Phi^*},\ \  H_\Phi^*= H_{\widetilde
\Phi}\quad (\Phi\in L^\infty_{M_n});\notag\\
&T_{\Phi\Psi}-T_\Phi T_\Psi = H_{\Phi^*}^*H_\Psi \quad (\Phi,\Psi\in L^\infty_{M_n});\label{1.5}\\
&H_\Phi T_\Psi = H_{\Phi\Psi},\ \ H_{\Psi\Phi}=T_{\widetilde{\Psi}}^*H_\Phi\quad (\Phi\in
L^\infty_{M_n}, \Psi\in H^\infty_{M_n});\label{1.5-1}\\
&H_\Phi^* H_\Phi - H_{\Theta \Phi}^* H_{\Theta\Phi} = H_\Phi^* H_{\Theta^*}H_{\Theta^*}^*H_\Phi \quad
(\Theta\in H^\infty_{M_n}\ \hbox{is inner},\ \Phi\in L^\infty_{M_n}).\notag
\end{align}

\noindent
For a matrix-valued function $\Phi=[\phi_{ij}]\in L^\infty_{M_n}$, we say that
$\Phi$ is of {\it bounded type} if each entry $\phi_{ij}$ is of
bounded type and that $\Phi$ is {\it rational}
if each entry $\phi_{ij}$ is a rational function.\
The {\it shift} operator $S$ on $H^2_{\mathbb C^n}$ is defined by
$$
S:=\sum_{j=1}^n \bigoplus T_z.
$$
The following fundamental result known as the Beurling-Lax-Halmos Theorem
is useful in the sequel.
\bigskip

\noindent
{\bf The Beurling-Lax-Halmos Theorem.}
{\it
A subspace $M$ of $H^2_{\mathbb C^n}$ is invariant under the shift
operator $S$ on $H^2_{\mathbb C^n}$ if and only if
$M=\Theta H^2_{\mathbb C^m}$, where
$\Theta$ is an inner matrix function in $H^{\infty}_{M_{n\times m}}$
($m\le n$).
}
\bigskip

\noindent
In view of (\ref{1.5-1}), the kernel of a block Hankel operator $H_\Phi$ is an invariant
subspace of the shift operator on $H^2_{\mathbb C^n}$.\
Thus if $\hbox{ker}\, H_\Phi\ne \{0\}$ then by the Beurling-Lax-Halmos Theorem,
$$
\hbox{ker}\, H_\Phi=\Theta H^2_{\mathbb{C}^m}
$$
for some inner matrix function $\Theta$. But we don't guarantee that
$\Theta$ is a square matrix. \
In fact, as we will refer in the sequel,
$\Theta$ is square if and only if $\Phi$ is of bounded type.\
Recently, Gu, Hendricks and Rutherford \cite{GHR} considered the
hyponormality of block Toeplitz operators and characterized the
hyponormality of block Toeplitz operators in terms of their symbols.\
In particular they showed that if $T_\Phi$ is a hyponormal block
Toeplitz operator on $H^2_{\mathbb{C}^n}$, then $\Phi$ is normal,
i.e., $\Phi^*\Phi=\Phi\Phi^*$.\
Their characterization for
hyponormality of block Toeplitz operators resembles the Cowen's
theorem except for an additional condition -- the normality
condition of the symbol.
\bigskip

\noindent
{\bf Hyponormality of Block Toeplitz Operators} (Gu-Hendricks-Rutherford \cite{GHR})
{\it
For each $\Phi\in L^\infty_{M_n}$, let
$$
\mathcal{E}(\Phi):=\Bigl\{K\in H^\infty_{M_n}:\
||K||_\infty \le 1\ \ \hbox{and}\ \ \Phi-K \Phi^*\in
H^\infty_{M_n}\Bigr\}.
$$
Then $T_\Phi$ is hyponormal if and
only if $\Phi$ is normal and $\mathcal{E}(\Phi)$ is nonempty.
}
\bigskip

For a matrix-valued function $\Phi\in H^2_{M_{n\times r}}$, we say
that $\Delta\in H^2_{M_{n\times m}}$ is a {\it left inner divisor}
of $\Phi$ if $\Delta$ is an inner matrix function such that
$\Phi=\Delta A$ for some $A \in H^{2}_{M_{m\times r}}$ ($m\le n$).\
We also say that two matrix functions $\Phi\in H^2_{M_{n\times r}}$
and $\Psi\in H^2_{M_{n\times m}}$ are {\it left coprime} if the only
common left inner divisor of both $\Phi$ and $\Psi$ is a unitary
constant and that $\Phi\in H^2_{M_{n\times r}}$ and $\Psi\in
H^2_{M_{m\times r}}$ are {\it right coprime} if $\widetilde\Phi$ and
 $\widetilde\Psi$ are left coprime.\
Two matrix functions $\Phi$ and $\Psi$ in $H^2_{M_n}$ are said to be
{\it coprime} if they are both left and right coprime.\
We remark
that if $\Phi\in H^2_{M_n}$ is such that $\hbox{det}\,\Phi$ is not
identically zero then any left inner divisor $\Delta$ of $\Phi$ is
square, i.e., $\Delta\in H^2_{M_n}$. If $\Phi\in H^2_{M_n}$ is such
that $\hbox{det}\,\Phi$ is not identically zero then we say that
$\Delta\in H^2_{M_{n}}$ is a {\it right inner divisor} of $\Phi$ if
$\widetilde{\Delta}$ is a left inner divisor of $\widetilde{\Phi}$.

\medskip

The following lemma will be useful in the sequel.

\smallskip

\noindent
\begin{lem}\label{lem1.1}(\cite{GHR})\
For $\Phi\in L^\infty_{M_n}$, the following statements are equivalent:
\medskip

{\rm (i)} $\Phi$ is of bounded type;

{\rm (ii)} $\hbox{ker}\, H_\Phi=\Theta H^2_{\mathbb{C}^n}$
for some square inner matrix function $\Theta$;

{\rm (iii)} $\Phi=A\Theta^*$, where $A\in H^{\infty}_{M_n}$ and $A$ and $\Theta$ are
right coprime.
\end{lem}

\medskip

\noindent
For $\Phi\in L^\infty_{M_n}$ we write
$$
\Phi_+:=P_n \Phi \in H^2_{M_n} \quad\hbox{and}\quad
\Phi_-:=\bigl(P_n^\perp \Phi\bigr)^* \in H^2_{M_n}.
$$
Thus we can write $\Phi=\Phi_-^*+\Phi_+\,$.
For an inner matrix function $\Theta\in H^\infty_{M_n}$,
write
$$
\mathcal{H}_{\Theta}:=\left(\Theta H^2_{\mathbb{C}^n}\right)^\perp
\equiv H^2_{\mathbb{C}^n}\ominus \Theta H^2_{\mathbb{C}^n}.
$$
Suppose $\Phi=[\phi_{ij}] \in L^\infty_{M_n}$ is such that $\Phi^*$ is of
bounded type.\
Then we may write
$\phi_{ij}=\theta_{ij}\overline{b}_{ij}$, where $\theta_{ij}$ is an
inner function and $\theta_{ij}$ and
$b_{ij}$ are coprime.\
Thus if $\theta$ is the least common multiple
of $\theta_{ij}$'s (i.e., the $\theta_{ij}$
divide $\theta$ and if they divide an inner function $\theta^\prime$
then $\theta$ in turn divides $\theta^\prime$),
then we can write
\begin{equation}\label{1.9}
\Phi=[\phi_{ij}]=[\theta_{ij}\overline{b}_{ij}]=[\theta
\overline{a}_{ij}]= \Theta A^* \quad (\Theta=\theta I_n,\ A \in H^{2}_{M_n}).
\end{equation}
We note that the representation (\ref{1.9}) is ``minimal," in the
sense that if $\omega I_n$ ($\omega$ is inner) is a common inner
divisor of $\Theta$ and $A$, then $\omega$ is constant. Let
$\Phi\equiv \Phi_-^*+\Phi_+\in L^\infty_{M_n}$ be such that $\Phi$
and $\Phi^*$ are of bounded type.\ Then in view of (\ref{1.9}) we
can write
$$
\Phi_+= \Theta_1 A^* \quad\hbox{and}\quad \Phi_-= \Theta_2 B^*,
$$
where $\Theta_i =\theta_i I_n$ with an inner function $\theta_i$ for $i=1,2$ and $A,B\in
H^{2}_{M_n}$.\
In particular, if $\Phi\in L^\infty_{M_n}$ is rational then the $\theta_i$ are chosen
as finite Blaschke products as we observed in (\ref{1.4}).
\bigskip

We would remark that, in (\ref{1.9}),
by contrast with scalar-valued functions, $\Theta$ and $A$ need not be (right) coprime:
indeed, if $\Phi:=\left[\begin{smallmatrix} z&z\\ z&z\end{smallmatrix}\right]$ then
we can write
$$
\Phi=\Theta A^* = \begin{bmatrix} z&0\\ 0&z\end{bmatrix}
\begin{bmatrix} 1&1\\ 1&1\end{bmatrix},
$$
but $\Theta:=\left[\begin{smallmatrix} z&0\\ 0&z\end{smallmatrix}\right]$
and $A:=\left[\begin{smallmatrix} 1&1\\ 1&1\end{smallmatrix}\right]$
are not right coprime because
$\frac{1}{\sqrt{2}} \left[\begin{smallmatrix} z&-z\\ 1&1\end{smallmatrix}\right]$
is a common right inner factor, i.e.,
\begin{equation}\label{1.10}
\Theta=\frac{1}{\sqrt{2}} \begin{bmatrix} 1&z\\ -1&z\end{bmatrix}
\,\cdot\, \frac{1}{\sqrt{2}} \begin{bmatrix} z&-z\\ 1&1\end{bmatrix}\quad\hbox{and}\quad
A=\sqrt{2} \begin{bmatrix} 0&1\\ 0&1\end{bmatrix}
\,\cdot\, \frac{1}{\sqrt{2}} \begin{bmatrix} z&-z\\ 1&1\end{bmatrix}.
\end{equation}

\bigskip

In this paper we consider the subnormality of
block Toeplitz operators and in particular, the block version of
Halmos's Problem 5: Which subnormal block Toeplitz operators are
either normal or analytic\,? \
In 1976, M. Abrahamse showed that
if $\phi=\overline g + f\in L^\infty$ ($f,g\in H^2$) is such that
$\phi$ or $\overline \phi$ is of bounded type, if $T_\phi$ is
hyponormal, and if $\hbox{\rm ker}\, [T_\phi^*, T_\phi]$ is
invariant under $T_\phi$ then $T_\phi$ is normal or analytic.\
The purpose of this paper is to establish an extension of Abrahamse's theorem
for block Toeplitz operators. \
In Section 2 we make a brief sketch on Halmos's Problem 5 and Abrahamse's theorem.\
Section 3 is devoted to the proof of the main result.\
In Section 4 we consider the scalar Toeplitz operators with finite
rank self-commutators.

\bigskip

\noindent
{\it Acknowledgment.}
The authors are deeply indebted to the referee for many
helpful comments that helped improved the presentation and mathematical content of the paper.

\vskip 1cm

%%%%%%%%%%%%%%%%%%%%%%%%%%%%%%%%%%%%%%%%%%%%%%%%%%%%%%%%%%%%%%%%%%%%%%%%%%%%%%%%%%%%%%%%%%%%%%%%%%
%
%                SECTION 2
%
%%%%%%%%%%%%%%%%%%%%%%%%%%%%%%%%%%%%%%%%%%%%%%%%%%%%%%%%%%%%%%%%%%%%%%%%%%%%%%%%%%%%%%%%%%%%%%%%%%%%

\section{Halmos's Problem 5 and Abrahamse's theorem}

\noindent
In 1970, P.R. Halmos posed the following
problem, listed as Problem 5 in his lectures ``Ten
problems in Hilbert space" \cite{Hal1}, \cite{Hal2}:
$$
\hbox{Is every subnormal Toeplitz operator either normal or analytic\,?}
$$
A Toeplitz operator $T_\phi$ is called {\it analytic}
if $\phi\in H^\infty$. Any analytic Toeplitz
operator is easily seen to be subnormal: indeed,
$T_\phi h=P(\phi h)=\phi h =M_\phi h$ for $h\in H^2$,
where $M_\phi$ is the normal operator of multiplication by
$\phi$ on $L^2$.\
The question is natural because the two classes, the normal and
analytic Toeplitz operators, are fairly well understood and are
subnormal.\
Halmos's Problem 5 has been partially answered in the
affirmative by many authors (cf. \cite{Ab}, \cite{AIW},
\cite{CuL1}, \cite{CuL2}, \cite{NT}, and etc).\
In 1984, Halmos's
Problem 5 was answered in the negative by C. Cowen and J. Long
\cite{CoL}: they found an analytic function $\psi$ for which
$T_{\psi+\alpha\overline\psi}$ ($0<\alpha<1$) is subnormal
- in fact, this Toeplitz operator is unitarily equivalent to
a subnormal weighted shift $W_\beta$ with weight sequence
$\beta\equiv\{\beta_n\}$, where
$\beta_n=(1-\alpha^{2n+2})^{\frac{1}{2}}$ for $n=0,1,2,\hdots$.\
Unfortunately, Cowen and Long's construction does not provide an intrinsic
connection between subnormality and the theory of Toeplitz operators. Until now
researchers have been unable to characterize subnormal Toeplitz operators
in terms of their symbols.\
On the other hand, surprisingly, as C. Cowen notes
in \cite{Co1} and \cite{Co2}, some analytic Toeplitz operators
are unitarily equivalent to non-analytic Toeplitz operators;
i.e., the analyticity of Toeplitz operators is not invariant
under unitary equivalence.\
In this sense, we might ask whether
Cowen and Long's non-analytic subnormal Toeplitz operator is unitarily
equivalent to an analytic Toeplitz operator.\
To this end, we have:
\medskip

\noindent
{\bf Observation.}
{\it Cowen and Long's non-analytic subnormal Toeplitz operator $T_\phi$
is not unitarily equivalent to any analytic Toeplitz operator.}
\medskip

\noindent{\it Proof.}
Assume to the contrary that $T_\phi$ is unitarily equivalent to
an analytic Toeplitz operator $T_f$. \
Then by the above remark, $T_f$ is
unitarily equivalent to the subnormal weighted shift $W_\beta$ with
weight sequence $\beta\equiv\{\beta_n\}$, where
$\beta_n=(1-\alpha^{2n+2})^{\frac{1}{2}}$ ($0<\alpha<1$) for $n=0,1,2,\hdots$; i.e.,
there exists a unitary operator $V$ such that
$$
V^*T_f V=W_\beta.
$$
Thus if $\{e_n\}$ is the canonical orthonormal basis for $\ell^2$ then
$$
V^*T_f V e_j=W_\beta e_j =\beta_j e_{j+1}\quad\hbox{for}\ j=0,1,2,\hdots.
$$
We thus have
$$
\bigl(V^*T_{|f|^2}V\bigr) e_j =W_\beta^*W_\beta e_j = \beta_j^2 e_j,
$$
and hence,
$$
T_{|f|^2-\beta_j^2} (Ve_j)=0\quad\hbox{for}\ j=0,1,2,\hdots.
$$
Fix $j\ge 0$ and observe that $Ve_j\in \hbox{ker}\,(T_{|f|^2-\beta_j^2})$.
By Coburn's Theorem, if $|f|^2-\beta_j^2$ is nonzero then
either $T_{|f|^2-\beta_j^2}$ or $T^*_{|f|^2-\beta_j^2}$ is
one-one.\
It follows that
$|f|^2=\beta_j^2$ for $j=0,1,2,\hdots$.\
This readily implies that $\beta_0=\beta_1=\beta_2=\cdots$, a contradiction.
\hfill$\square$
\bigskip

Consequently, even if we interpret ``{\sl is}" in Halmos Problem 5
as ``{\sl is up to unitary equivalence}," the answer to Halmos Problem 5
is still negative.

\bigskip

We would like to reformulate Halmos's Problem 5 as follows:
\medskip

\noindent
{\bf Halmos's Problem 5 reformulated.} {\it
Which Toeplitz operators are subnormal\,?}

\bigskip

The most interesting partial answer to Halmos's Problem 5 was given by
M. Abrahamse \cite{Ab}. \
M. Abrahamse gave a general sufficient condition for
the answer to Halmos's Problem 5 to be affirmative.

\medskip

Abrahamse's theorem can be then stated as:
\bigskip

\noindent{\bf Abrahamse's Theorem} (\cite[Theorem]{Ab}).
{\it
Let $\phi=\overline{g}+f\in L^\infty$ ($f,g\in H^2$) be such that $\phi$ or
$\overline \phi$ is of bounded type.
If $T_\phi$ is hyponormal and $\hbox{\rm ker}\,[T_\phi^*, T_\phi]$
is invariant under $T_\phi$ then $T_\phi$ is normal or analytic.
}
\bigskip

Consequently, if $\phi=\overline{g}+f\in L^\infty$ ($f,g\in
H^2$) is such that $\phi$ or $\overline \phi$ is of bounded
type, then every subnormal Toeplitz operator must be normal or
analytic.

We say that a block Toeplitz operator $T_\Phi$ is {\it analytic} if
$\Phi\in H^\infty_{M_n}$.\
Evidently, any analytic block Toeplitz operator with a normal symbol
is subnormal because the multiplication operator $M_\Phi$ is
a normal extension of $T_\Phi$.\
As a first inquiry in the above reformulation of Halmos's Problem 5
the following question can be raised:
$$
\hbox{Is Abrahamse's Theorem valid for block Toeplitz operators}\,?
$$

In this paper we answer this question in the affirmative (Theorem \ref{thm3.7}).

\vskip 1cm

%%%%%%%%%%%%%%%%%%%%%%%%%%%%%%%%%%%%%%%%%%%%%%%%%%%%%%%%%%%%%%%%%%%%%%%%%%%%%%%%%%%%%%%%%%%%%%%%%%
%
%                SECTION 3
%
%%%%%%%%%%%%%%%%%%%%%%%%%%%%%%%%%%%%%%%%%%%%%%%%%%%%%%%%%%%%%%%%%%%%%%%%%%%%%%%%%%%%%%%%%%%%%%%%%%%%

\section{Abrahamse's Theorem for matrix-valued symbols}

\noindent
Recall the representation (\ref{1.9}), and for $\Psi\in L^\infty_{M_n}$
such that $\Psi^*$ is of bounded type, write
$\Psi=\Theta_2B^*=B^*\Theta_2$. \
Let $\Omega$ be the greatest common left inner divisor of $B$
and $\Theta_2$. Then $B=\Omega B_{\ell}$ and $\Theta_2=\Omega
\Omega_2$ for some $B_\ell \in H^{2}_{M_n}$
and some inner matrix $\Omega_2$. \
Therefore we can write
$$
\Psi={B^*_{\ell}}\Omega_2,\quad\hbox{where $B_\ell$ and $\Omega_2$ are left coprime:}
$$
in this case, $B_\ell^*\Omega_2$ is called a {\it left coprime factorization} of $\Psi$.\
Similarly,
$$
\Psi=\Delta_2{B^*_{r}}, \quad\hbox{where $B_r$ and $\Delta_2$ are right coprime:}
$$
in this case, $\Delta_2B_r^*$ is called a {\it right coprime factorization} of $\Psi$.

\bigskip

To prove our main result (Theorem \ref{thm3.7}), we need several auxiliary lemmas.

We begin with:
\medskip

\noindent
\begin{lem}\label{lem3.1}
Suppose $\Phi = \Phi_-^*+ \Phi_+ \in L^{\infty}_{M_n}$ is
such that $\Phi$ and $\Phi^*$ are of bounded type of the form
$$
\Phi_+ =A^* \Theta_1 \quad \hbox{and} \quad \Phi_- = B^*\Theta_2\,,
$$
where $\Theta_i:=\theta_i I_n$  with an inner function $\theta_i \ (i=1,2)$.\
If $T_{\Phi}$ is hyponormal, then $\Theta_2$ is a right inner divisor of $\Theta_1$.
\end{lem}

\begin{proof}
Suppose $T_\Phi$ is hyponormal.\
Then there
exists a matrix function $K \in H^{\infty}_{M_n}$ such that
$\Phi_-^*-K \Phi_+^* \in H^{2}_{M_n}$.\
Thus $B\Theta_2^*-KA\Theta_1^*=F$ for some $F\in H^{2}_{M_n}$,
which implies that
$B\Theta_2^*\Theta_1\in H^2_{M_n}$.\
Now we write
$\Phi_-=\begin{bmatrix} f_{ij}\end{bmatrix}_{n\times n}$.\
Since $\Phi$ is of bounded type we can write
$f_{ij}=\theta_{ij}\overline{c}_{ij}$, where $\theta_{ij}$ is an
inner function, $c_{ij}$ is in $H^2$, and
$\theta_{ij}$ and $c_{ij}$ are coprime.\
Write $B=\begin{bmatrix} b_{ij}\end{bmatrix}_{n\times n}$.\
We thus have
$$
f_{ij}=\theta_{ij}\overline{c}_{ij}=\theta_2
\overline{b}_{ji}\quad\hbox{for each}\ i,j=1,\hdots, n,
$$
which implies that $b_{ji}=\overline{\theta}_{ij}\theta_2c_{ij}$.\
But since $B\Theta_2^*\Theta_1=[\theta_1
\overline{\theta}_2 b_{ij}] \in H^{2}_{M_n}$, we have
$\theta_1 \overline{\theta}_{ji} c_{ji} \in H^{2}$.\
Since
$\theta_{ji}$ and $c_{ji}$ are coprime for each
$i,j=1,\cdots ,n$, it follows that
$\overline{\theta}_{ji}\theta_1 \in H^{2}$,
which implies that $\overline{\theta}_2\theta_1\in H^{2}$
and therefore,
$\Theta_2$ divides $\Theta_1$, i.e.,
$\Theta_1=\Theta_0\Theta_2$ for some inner matrix function $\Theta_0$.
\end{proof}

\bigskip

In the sequel, when we consider the symbol
$\Phi=\Phi_-^* + \Phi_+\in L^{\infty}_{M_n}$, which is such that
$\Phi$ and $\Phi^*$ are of bounded type and for which
$T_\Phi$ is hyponormal,
we will, in view of Lemma 3.1, assume that
$$
\Phi_+=  A^* \Omega_1\Omega_2\quad \hbox{and} \quad \Phi_-=
B_\ell^*\Omega_2\ \hbox{(left coprime factorization)},
$$
where $\Omega_1 \Omega_2=\Theta=\theta I_n$.
We also note that $\Omega_2 \Omega_1=\Theta$:
indeed, if $\Omega_1\Omega_2=\Theta=\theta I_n$, then
$(\overline\theta I_n \Omega_1)\Omega_2=I_n$, so that
$\Omega_1(\overline\theta I_n \Omega_2)=I_n$, which implies that
$(\overline\theta I_n \Omega_2)\Omega_1=I_n$, and hence
$\Omega_2\Omega_1=\theta I_n=\Theta$.

\bigskip

\noindent
\begin{lem}\label{lem3.2}
Suppose
$\Phi = \Phi_-^*+ \Phi_+ \in L^{\infty}_{M_n}$ is such that $\Phi$
and $\Phi^*$ are of bounded type of the form
$$
\Phi_+ =\Delta_1 A_r^*\ \hbox{\rm (right coprime factorization)}
\quad \hbox{and} \quad \Phi_- = \Delta_2 B_r^*\ \hbox{\rm (right
coprime factorization)}.
$$
If $T_{\Phi}$ is hyponormal, then $\Delta_2$ is a left inner divisor
of $\Delta_1$, i.e., $\Delta_1=\Delta_2\Delta_0$ for some
$\Delta_0$.
\end{lem}

\begin{proof}
Suppose $T_{\Phi}$ is hyponormal. Then there exists $K\in
H^\infty_{M_n}$ such that $\Phi-K\Phi^*\in H^\infty_{M_n}$. Thus
$H_\Phi=H_{K\Phi^*}=T_{\widetilde K}^*H_{\Phi^*}$, which implies
that $\hbox{ker}H_{\Phi_+^*} \subseteq \hbox{ker}\, H_{\Phi_-^*}$,
so that by Lemma \ref{lem1.1},
$\Delta_1 H^2_{\mathbb C^n} \subseteq \Delta_2 H^2_{\mathbb
C^n}$. It follows (cf. [FF, Corollary IX.2.2]) that $\Delta_2$ is a
left inner divisor of $\Delta_1$.
\end{proof}

\bigskip

On the other hand, the condition ``(left/right) coprime factorization"
is not so easy to check in
general. For example, consider a simple case:
$\Phi_-:=\left[\begin{smallmatrix} z&z\\
z&z\end{smallmatrix}\right]$.
One is tempted to write
$$
\Phi_-:=\left[\begin{matrix} z&0\\ 0&z\end{matrix}\right]\,
\left[\begin{matrix} 1&1\\ 1&1\end{matrix}\right]^\ast.
$$
But $\left[\begin{smallmatrix} z&0\\ 0&z\end{smallmatrix}\right]$
and $\left[\begin{smallmatrix} 1&1\\ 1&1\end{smallmatrix}\right]$ are not
right coprime as we have seen in the Introduction.
On the other hand, observe that
$$
\left[\begin{matrix} 1&1\\ 1&1\end{matrix}\right] \equiv \Delta B^*,
$$
where
$$
\Delta:=\frac{1}{\sqrt{2}} \left[\begin{matrix} 1&z\\ -1&z\end{matrix}\right]\
\hbox{is inner and}\ B:=\frac{1}{\sqrt{2}}
\left[\begin{matrix} 0&2z\\ 0&2z\end{matrix}\right].
$$
Again, $\Delta$ and $B$ are not right coprime because
$\hbox{ker}\, H_{\left[\begin{smallmatrix} 1&1\\
1&1\end{smallmatrix}\right]}=H^2_{\mathbb C^2}$.
Thus we might choose
$$
\Phi_-=\bigl(zI_2\,\Delta\bigr)\cdot B^*\quad\hbox{or}\quad
\Phi_-=\Delta\cdot \bigl(\overline z I_2\, B\bigr)^*.
$$
A straightforward calculation show that $\hbox{ker}\, H_{\Phi_-^*}
=\Delta H^2_{\mathbb C^2}$.
Hence the latter of the above factorizations is the desired factorization:
i.e., $\Delta$ and $\overline z I_2\, B$ are right coprime.

\bigskip

However, if $\Theta$ is given in a form $\Theta=\theta I_n$
with a finite Blaschke product $\theta$,
then we can obtain a more tractable criterion on the coprime-ness of $\Theta$
and $B\in H^2_{M_n}$.
To see this, recall that an $n\times n$
matrix-valued function $D$ is called {\it a finite
Blaschke-Potapov product} if $D$ is of the form
\begin{equation*}
%%\label{3.18}
D(z)=\nu \prod_{m=1}^M \Bigl(b_m(z)P_m + (I-P_m)\Bigr),
\end{equation*}
where $\nu$ is an $n\times n$ unitary constant matrix, $b_m$ is a
{\it Blaschke factor}, which is of the form
$$
b_m(z):=\frac{z-\alpha_m}{1-\overline\alpha_m z}\quad
\hbox{($\alpha_m\in \mathbb D$)},
$$
and $P_m$ is an orthogonal projection in $\mathbb C^n$. In
particular, a scalar-valued function $D$ reduces to a finite
Blaschke product $D(z)=\nu \prod_{m=1}^M b_m(z)$, where
$\nu=e^{i\omega}$. It was known [Po] that an $n\times n$
matrix-valued function $D$ is rational and inner if and only if it
can be represented as a finite Blaschke-Potapov product.

\bigskip

We write $\mathcal{Z}(\theta)$ for the set of zeros of an inner function $\theta$.
We then have:
\medskip

\begin{lem}\label{lem3.12}
Let $B\in H^2_{M_n}$ and $\Theta:=\theta I_n$ with a
finite Blaschke product $\theta$. Then the following statements are
equivalent:
\medskip

{\rm (a)} $B(\alpha)$ is invertible for each $\alpha \in \mathcal{Z}(\theta)$;

{\rm (b)} $B$ and  $\Theta$ are right coprime;

{\rm (c)} $B$ and  $\Theta$ are left coprime.
\end{lem}

\begin{proof}
We first write
$$
\theta(z) = e^{i \xi} \prod_{i=1}^N \Bigl(\frac{z- \alpha_i}{1-
\overline{\alpha_i}z} \Bigr)^{m_i}\qquad (\sum_{i=1}^{N} m_i=:d).
$$
\medskip

\noindent (a) $\Leftrightarrow$ (b): Suppose $B(\alpha)$ is
invertible for each $\alpha \in \mathcal Z(\theta)$.
Assume to the contrary that $B$ and $\Theta$ are not right
coprime. Thus there exists a finite Blaschke-Potapov product $D$ of
the form
$$
D(z)=\nu \prod_{m=1}^M \Bigl(b_m(z)P_m + (I-P_m)\Bigr)
$$
satisfying that
$$
B=B_1 D\quad\hbox{and}\quad \Theta=\Theta_0 D\quad\hbox{for some
inner function}\ \Theta_0.
$$
Thus if $\alpha \in \mathcal Z(b_{m_0})$ for some $1\le m_0\le M$,
then $\Theta(\alpha)=\Theta_0(\alpha)D(\alpha)$ is not invertible.
But since $\Theta=\theta I_n$, it follows that $\Theta (\alpha)=0$
and hence $\alpha \in \mathcal Z(\theta)$. Moreover,
$$
\begin{aligned}
\hbox{det}\,B(\alpha)=
\hbox{det}\,B_1(\alpha)\,\hbox{det}\,D(\alpha) &=\hbox{det}(\nu)\,
\hbox{det}\,B_1(\alpha) \prod_{m=1}^M \hbox{det}\,\bigl(b_m(\alpha)
P_m+(I-P_m)\bigr)=0,
\end{aligned}
$$
giving a contradiction. Therefore $B$ and $\Theta$ are right
coprime.

For the converse we assume that $B(\alpha_{i_0})$ is not invertible
for some $i_0$. Then the following matrix is not invertible:
$$
\mathcal B:=\begin{bmatrix} B_{0}&0&0&0&\cdots&0\\
B_{1}&B_{0}&0&0&\cdots&0\\
B_{2}&B_{1}&B_{0}&0&\cdots&0\\
\vdots&\ddots&\ddots&\ddots&\ddots&\vdots \\
B_{m_{i_0} -2}&B_{m_{i_0} -3}&\ddots&\ddots&B_{0}&0\\
B_{m_{i_0} -1}&B_{m_{i_0} -2}&\hdots&B_{2}&B_{1}&B_{0}
\end{bmatrix}
\quad \Bigl(B_j:=\frac{B^{(j)}(\alpha_{i_0})}{j!} \Bigr).
$$
Thus there exists a nonzero $n\times m_{i_0}$ matrix $\mathcal
G=\begin{pmatrix} \mathcal G_0 \ \mathcal G_1 \ \cdots \ \mathcal
G_{m_{i_0}-1} \end{pmatrix}^t$ such that $\mathcal B \mathcal G =0.$
We now want to show that there exists $\mathfrak{h}=\begin{pmatrix}
h_1 \ h_2 \ \cdots \ h_n
\end{pmatrix}^t \in H^2_{\mathbb{C}^n}$ satisfying the following
property:
\begin{equation}\label{3-18-1}
\frac{\mathfrak{h}^{(j)}(\alpha_i)}{j!}=\begin{cases} \mathcal G_j \quad &(i = i_0)\\
0 \ &(i \neq i_0) \end{cases}.
\end{equation}
This is exactly the classical Hermite-Fej\' er interpolation problem
(cf. [FF]), so that we use an argument of a solution for the
interpolation of this type. Thus we can construct a function (in
fact, a polynomial) $\mathfrak{h}(z)\equiv P(z)$ satisfying
(\ref{3-18-1}) (see [FF, p.299]). Then $P(z)$ belongs to
$\text{ker}H_{B \Theta^*}$. Since
$$
\mathcal G=\begin{bmatrix} \mathcal G_0 \ \mathcal G_1 \ \cdots \
\mathcal G_{m_{i_0}-1}
\end{bmatrix}^t \neq 0,
$$
it follows that $P(z) \notin \Theta H^2_{\mathbb{C}^n}$. Therefore
we have $\hbox{\rm ker}\,H_{B \Theta^*} \neq \Theta
H^2_{\mathbb{C}^n}$, which implies that $B$ and $\Theta$ are not
right coprime.

\medskip

\noindent (b) $\Leftrightarrow$ (c): Suppose  $B$ and $\Theta$ are
right coprime. If $B$ and  $\Theta$ are not left coprime, there
exists a nonconstant inner matrix $\Delta \in H^2_{M_n}$ such that
$B=\Delta B_1$ and $\Theta=\Delta \Omega.$
We thus have that for each $i=1,2,\cdots,N$
$$
\begin{bmatrix} \Delta_{i,0}&0&0&0&\cdots&0\\
\Delta_{i,1}&\Delta_{i,0}&0&0&\cdots&0\\
\Delta_{i,2}&\Delta_{i,1}&\Delta_{i,0}&0&\cdots&0\\
\vdots&\ddots&\ddots&\ddots&\ddots&\vdots \\
\Delta_{i,m_i -2}&\Delta_{i,m_i -3}&\ddots&\ddots&\Delta_{i,0}&0\\
\Delta_{i,m_i -1}&\Delta_{i,m_i
-2}&\hdots&\Delta_{i,2}&\Delta_{i,1}&\Delta_{i,0}
\end{bmatrix} \begin{bmatrix}
\Omega_{i,0}\\
\Omega_{i,1}\\
\Omega_{i,2}\\
\vdots\\
\Omega_{i,m_i -2}\\
\Omega_{i,m_i -1}
\end{bmatrix} =0,
$$
where
$$
\Delta_{i,j}:= \frac{\Delta^{(j)}(\alpha_i)}{j!}\quad\hbox{and}\quad
\Omega_{i,j}:= \frac{\Omega^{(j)}(\alpha_i)}{j!}.
$$
But since $B(\alpha)$ is invertible for each $\alpha\in
\mathcal{Z}(\theta),$ we have that $\Delta_{i,0}$ is invertible for
each $i=1,2,\cdots,N$. Thus
$$
\Omega_{i,j}=0 \quad (i=1,2,\cdots,N, \ j=0,1,2,\cdots, m_{i}-1),
$$
which implies that $\Omega = \Theta \Omega_1$ for some $\Omega_1 \in
H^2_{M_n}$. Thus $\Theta = \Delta \Omega=\Delta \Theta \Omega_1$, so
that $I=\Delta \Omega_1$ and hence $\Delta^*=\Omega_1$, which
implies that $\Delta$ is a constant matrix, a contradiction. Thus
$B$ and  $\Theta$ are left coprime. The converse follows from the
same argument. This completes the proof.
\end{proof}

%%%%%%%%%%%%%%%%%%%%%%%%%%%%%%%%%%%%%%%%%%%%%%%%%%%%%%%%%%%%%%%%%%%%%%%%%%%%%%%%%%%%%%%%%%%%%%%%%%%%%

\begin{lem}\label{lem3.6}
Let $\theta_0$ be a nonconstant inner function.
Then $\mathcal{H}_{\theta_0}$ contains an outer function that is invertible in $H^\infty$.
\end{lem}

\begin{proof}
If $\theta_0$ has at least one
Blaschke factor, say $\frac{z-\alpha}{1-\overline\alpha z}$
($|\alpha|<1$), then $\frac{1}{1-\overline\alpha z}$ is an outer
function and $\frac{1}{1-\overline\alpha z}\in \mathcal H_{\theta_0}$
because $\frac{1}{1-\overline\alpha z}$ is the reproducing kernel
for $\alpha$, so that for any $f\in H^2$,
$$
\left\langle \theta_0 f ,\, \frac{1}{1-\overline\alpha z}
\right\rangle= \theta_0(\alpha)f(\alpha)=0.
$$
Now suppose $\theta_0$ is a nonconstant singular inner function of
the form
$$
\theta_0(z):= \hbox{exp}\left(-
\int_{-\pi}^{\pi}\frac{e^{i\theta}+z}{e^{i \theta}-z}\,d\mu(\theta)
\right),
$$
where $\mu$ is a finite positive Borel measure on $\mathbb T$ which
is singular with respect to Lebesgue measure. We put
$$
\omega (z):= \hbox{exp}\left(-
\int_{-\pi}^{\pi}\frac{e^{i\theta}+z}{e^{i
\theta}-z}\,d\frac{\mu}{2}(\theta) \right).
$$
Then $\omega^2=\theta_0$. If $\alpha:=\overline\omega(0)$ then
evidently, $0<|\alpha|<1$ since $\omega$ is not constant. Note that
$\overline{\theta}_0\left(\omega-\frac{1}{\alpha}\right)
=\overline\omega-\frac{1}{\alpha}\overline{\theta}_0
\in \left(H^2\right)^{\perp}$, since
$\left(\overline\omega-\frac{1}{\alpha}\overline{\theta}_0\right)(0)
=\alpha-\frac{1}{\alpha}\alpha^2=0$. We thus have
$\omega-\frac{1}{\alpha}\in \mathcal H_{\theta_0}$. Also a
straightforward calculation shows that
$\frac{1}{\omega-\frac{1}{\alpha}}$ is bounded and analytic in
$\mathbb D$, which says that $\omega-\frac{1}{\alpha}$ is invertible
in $H^\infty$. Hence $\omega-\frac{1}{\alpha}$ is an outer function
in $\mathcal H_{\theta_0}$. This completes the proof.
\end{proof}

\bigskip

Before proving the main result,
we recall the inner-outer factorization of vector-valued functions.
If $D$ and $E$ are Hilbert spaces and if
$F$ is a function with values in $\mathcal{B}(E,D)$ such that
$F(\cdot)e\in H^2_{D}(\mathbb T)$ for each $e\in E$,
then $F$ is called a strong $H^2$-function.
The strong $H^2$-function $F$ is called
a ({\it strong}) {\it inner function}
if $F(\cdot)$ is a unitary operator from $D$ into $E$.
Write $\mathcal{P}_{E}$ for the set of all polynomials with values in $E$, i.e.,
$p(\zeta)=\sum_{k=0}^n \widehat{p}(k)\zeta^k$, $\widehat{p}(k)\in E$.
Then the function $Fp=\sum_{k=0}^n F\widehat{p}(k) z^k$
belongs to $H^2_{D}(\mathbb T)$. The strong $H^2$-function $F$ is called
{\it outer} if
$$
\hbox{cl}\, F\cdot\mathcal{P}_E=H^2_{D}(\mathbb T).
$$
Note that if $\hbox{dim}\, D=\hbox{dim}\, E=n<\infty$, then
evidently, every $F\in H^2_{M_n}$ is a strong $H^2$-function. We
then have an analogue of the scalar Inner-Outer Factorization
Theorem.
\bigskip

\noindent{\bf Inner-Outer Factorization.} (cf. [Ni])\quad
Every strong $H^2$-function $F$ with values in $\mathcal{B}(E, D)$
can be expressed in the form
$$
F=F^iF^e,
$$
where $F^e$ is an outer function with values in $\mathcal{B}(E, D^\prime)$ and
$F^i$ is an inner function with values in $\mathcal{B}(D^\prime,D)$
for some Hilbert space $D^\prime$.

\bigskip

We are now ready to prove the main result of this paper.
\smallskip

\begin{thm}({\bf Abrahamse's Theorem for Matrix-Valued Symbols})\label{thm3.7}
Suppose $\Phi:=\Phi_-^*+ \Phi_+\in L^\infty_{M_n}$ is such that
$\Phi$ and $\Phi^*$ are
of bounded type. In view of Lemma \ref{lem3.1}, we may write
$$
\Phi_+=A^*\Theta_0\Theta_2\quad\hbox{and}\quad
\Phi_- = B^*\Theta_2,
$$
where $\Theta_i =\theta_i I_n$ with an inner function $\theta_i$ ($i=0,2$)
and $A, B\in H^{2}_{M_n}$. Assume that $A,B$ and $\Theta_2$
are left coprime.
If
\medskip

{\rm (i)} $T_\Phi$ is hyponormal; and

{\rm (ii)} $\text{\rm ker}\,[T_{\Phi}^*, T_{\Phi}]$ is invariant under $T_{\Phi}$
\medskip

\noindent then $T_{\Phi}$ is normal or analytic.
Hence, in particular, if $T_\Phi$ is subnormal then it is normal or analytic.
\end{thm}

\medskip

\begin{rem}\label{rem3.8-1}
We note that if $n=1$ (i.e., $T_\Phi$ is a scalar Toeplitz operator)
then $\Phi_+=\overline a \theta_0\theta_2$ and
$\Phi_-=\overline{b}\theta_2$ with $a,b\in H^2$.
Thus, we can always arrange that $a,b$ and $\theta$ are coprime.
Consequently, if $n=1$ then our matrix version reduces to the original Abrahamse's Theorem.
\end{rem}

\smallskip

\begin{proof}
If $\Theta_2$ is constant then $\Phi_-=0$, so that $T_\Phi$ is analytic.
Suppose that $\Theta_2$ is nonconstant.

We split the proof into three steps.

\medskip

STEP 1: We first claim that
\begin{equation}\label{3.9}
\Theta_0 H^2_{\mathbb C^n}\ \subseteq\ \hbox{\rm ker}\,[T_{\Phi}^*,
T_{\Phi}].
\end{equation}
To see this, we observe that
\begin{equation}\label{3.9-1}
[T_{\Phi}^* , T_{\Phi}]=
H_{\Phi_+^*}^*H_{\Phi_+^*}-H_{\Phi_-^*}^*H_{\Phi_-^*}
=H_{A\Theta_2^*\Theta_0^*}^* H_{A\Theta_2^*\Theta_0^*} - H_{\Theta_2^*B}^* H_{\Theta_2^* B},
\end{equation}
which implies that
\begin{equation}\label{3.10}
\Theta_0\Theta_2 H^2_{\mathbb{C}^n}\, \subseteq\, \hbox{\rm ker}\,[T_{\Phi}^*,
T_{\Phi}].
\end{equation}
On the other hand,
since $\Theta_0\Theta_2$ is diagonal, we have that for all $g \in
\mathcal{P}_{\mathbb{C}^n}$,
$$
\aligned T_{\Phi}(\Theta_0\Theta_2 g)
&=P_n(\Theta_2^* B \Theta_0\Theta_2 g + \Phi_+ \Theta_0\Theta_2 g)\\
&=\Theta_0 B g +\Theta_0\Theta_2 \Phi_+ g\\
&=P_{\mathcal H_{\Theta_0\Theta_2}}( \Theta_0 B g)+P_{\Theta_0\Theta_2
H^2_{\mathbb{C}^n}} (\Theta_0 B g)+\Theta_0\Theta_2  \Phi_+ g.
\endaligned
$$
Since $\mathcal H_{\Theta_0 \Theta_2}=\mathcal
H_{\Theta_0} \bigoplus \Theta_0\mathcal H_{\Theta_2}$, it follows that
$$
P_{\mathcal H_{\Theta_0\Theta_2}}(\Theta_0 B g)=P_{\Theta_0\mathcal
H_{\Theta_2}}( \Theta_0 B g).
$$
We thus have
\begin{equation}\label{3.11}
T_{\Phi}(\Theta_0\Theta_2 g)=P_{\Theta_0\mathcal H_{\Theta_2}}(\Theta_0 B
g)+P_{\Theta_0\Theta_2 H^2_{\mathbb{C}^n}} (\Theta_0 B g)+\Theta_0\Theta_2 \Phi_+g.
\end{equation}
We claim that
\begin{equation}\label{3.12}
\mathcal H_{\Theta_2} = \hbox{cl}\,\Bigl\{P_{\mathcal H_{\Theta_2}}
(B g): g \in \mathcal{P}_{\mathbb{C}^n}\Bigr\}.
\end{equation}
In view of the above mentioned Inner-Outer Factorization, let $B=B^i B^e$ be the
inner-outer factorization of $B$ (as a strong $H^2$-function),
where $B^i\in H^\infty_{M_{n\times r}}$ and $B^e\in H^2_{M_{r\times n}}$.
Since $B$ and $\Theta_2$ are left coprime, $B^i$ and $\Theta_2$
are left coprime. Thus it follows from the Beurling-Lax-Halmos
Theorem that
$$
\Theta_2 H^2_{\mathbb{C}^n} \bigvee \hbox{cl}\,B
\mathcal{P}_{\mathbb{C}^n} =\Theta_2 H^2_{\mathbb{C}^n} \bigvee
B^i \bigl( \hbox{cl}\, B^e \mathcal{P}_{\mathbb{C}^n}\bigr)
=\Theta_2 H^2_{\mathbb{C}^n}\bigvee B^i H^2_{\mathbb{C}^r} =
H^2_{\mathbb{C}^n},
$$
giving (\ref{3.12}). Thus we have
\begin{equation}\label{3.13}
\Theta_0\mathcal H_{\Theta_2} =\hbox{cl}\, \Theta_0 \Bigl\{P_{
\mathcal H_{\Theta_2}}( B g): g \in
\mathcal{P}_{\mathbb{C}^n}\Bigr\}=\hbox{cl}\,  \Bigl\{P_{\Theta_0
\mathcal H_{\Theta_2}}( \Theta_0 B g): g \in
\mathcal{P}_{\mathbb{C}^n}\Bigr\}.
\end{equation}
If $\hbox{ker}\,[T_{\Phi}^*, T_{\Phi}]$ is invariant under
$T_{\Phi}$ then since $\text{ker}\,[T_{\Phi}^*, T_{\Phi}]$ is a
closed subspace it follows from (\ref{3.10}) -  (\ref{3.13}) that
$$
\Theta_0\mathcal H_{\Theta_2} \subseteq \text{ker}\,[T_{\Phi}^*,
T_{\Phi}].
$$
We thus have
$$
\Theta_0 H^2_{\mathbb C^n} =
\Theta_0\mathcal H_{\Theta_2}\ \bigoplus\ \Theta_0\Theta_2 H^2_{\mathbb C^n}
\subseteq \hbox{\rm ker}\,[T_{\Phi}^*, T_{\Phi}],
$$
which proves (\ref{3.9}).
\medskip

\medskip

STEP 2: We next claim that
\begin{equation}\label{3.13-1}
\hbox{$\mathcal{E}(\Phi)$
contains an inner function $K$.}
\end{equation}
To see this, we first observe that if $K\in\mathcal{E}(\Phi)$ then by (\ref{1.5-1}),
\begin{equation}\label{3.13-2}
[T_\Phi^*, T_\Phi]=H_{\Phi_+^*}^*H_{\Phi_+^*}-H_{K\Phi_+^*}^*H_{K\Phi_+^*}
=H_{\Phi_+^*}^* (I- T_{\widetilde K}T_{\widetilde K}^*) H_{\Phi_+^*}\,,
\end{equation}
so that
$$
\hbox{ker}\,[T_\Phi^*, T_\Phi]=\hbox{ker}\,(I- T_{\widetilde K}T_{\widetilde K}^*) H_{\Phi_+^*}.
$$
Thus by (\ref{3.9}),
$$
\{0\}=(I-T_{\widetilde{K}}T_{\widetilde{K}}^*)H_{A\Theta_2^*\Theta_0^*}(\Theta_0 H_{\mathbb C^n}^2)
= (I-T_{\widetilde{K}}T_{\widetilde{K}}^*)H_{A\Theta_2^*}(H_{\mathbb C^n}^2)\,,
$$
which implies
\begin{equation}\label{3.15}
\hbox{cl ran}\, H_{A\Theta_2^*}\subseteq \hbox{ker}\,(I-T_{\widetilde{K}}T_{\widetilde{K}}^*).
\end{equation}
Since by assumption, $A$ and $\Theta_2$ are left coprime, and hence
$\widetilde A$ and $\widetilde\Theta_2$ are are right coprime,
it follows from Lemma \ref{lem1.1} that
\begin{equation}\label{3.15-1}
\hbox{cl ran}\, H_{A\Theta_2^*}=\Bigl(\hbox{ker}\, H_{\widetilde A\widetilde \Theta_2^*}\Bigr)^{\perp}
=\Bigl(\widetilde\Theta_2 H^2_{\mathbb C^n}\Bigr)^\perp
=\mathcal{H}_{\widetilde\Theta_2},
\end{equation}
which together with (\ref{3.15}) implies
\begin{equation}\label{3.14-1}
\mathcal{H}_{\widetilde \Theta_2} \subseteq \hbox{ker}\,(I-T_{\widetilde{K}}T_{\widetilde{K}}^*).
\end{equation}
We thus have
\begin{equation}\label{3.14-2}
F=T_{\widetilde{K}}T_{\widetilde{K}}^* F\quad\hbox{for each $F\in\mathcal H_{\widetilde{\Theta}_2}$}.
\end{equation}
But since $||\widetilde K||_\infty=||K||_\infty \le 1$, we have
$$
||P_n(\widetilde K^*F)||_2\le ||\widetilde K^*F||_2\le ||F||_2
=||T_{\widetilde{K}}T_{\widetilde{K}}^* F||_2
=||\widetilde K P_n(\widetilde K^*F)||_2
\le ||P_n(\widetilde K^*F)||_2,
$$
which gives
$$
||P_n(\widetilde K^*F)||_2=||\widetilde K^*F||_2,
$$
which implies $\widetilde K^*F\in H^2_{\mathbb C^n}$. Therefore by (\ref{3.14-2}), we have
$$
F=\widetilde{K}\widetilde{K}^*F\quad\hbox{for each $F\in\mathcal H_{\widetilde{\Theta}_2}$}.
$$
In view of Lemma \ref{lem3.6}, we can choose an outer function $f\in\mathcal
H_{\widetilde{\theta}_2}$, which is invertible in $H^\infty$.
For each $j=1,2, \cdots, n$, we define
$$
F_j:=(0,\cdots, 0, f,0, \cdots,0)^t \quad\hbox{(where $f$ is the
$j$-th component)}.
$$
Then $F_j\in\mathcal H_{\widetilde{\Theta}_2}$ for each $j=1,2, \cdots, n$,
so that $(I-\widetilde{K}\widetilde{K}^*)F_j=0$ for each $j=1,2,\cdots,n$.
If we write $I-\widetilde{K}\widetilde{K}^*\equiv [q_{ij}]_{1\le i,j\le n}\in L^\infty_{M_n}$,
then
$q_{ij}f=0$ for each $i,j=1,2,\cdots,n$, so that $q_{ij}=0$ for each $i,j=1,2,\cdots,n$
because $f$ is invertible.
Therefore we have $K^*K=\widetilde{I}=I$,
which implies that $K$ is an inner function.
This proves (\ref{3.13-1}).

\medskip

STEP 3: Now since $K$ is inner it follows from (\ref{1.5}) that
$$
I-T_{\widetilde{K}}T_{\widetilde{K}}^*=T_{\widetilde{K}\widetilde{K}^*}-T_{\widetilde{K}}T_{\widetilde{K}}^*
=H_{\widetilde{K}^*}^*H_{\widetilde{K}^*}.
$$
Thus by (\ref{3.14-1}), we have
\begin{equation}\label{3.14-3}
\mathcal{H}_{\widetilde \Theta_2} \subseteq \hbox{ker}\,H_{\widetilde{K}^*}^* H_{\widetilde K^*}= \hbox{ker}\,H_{\widetilde K^*}.
\end{equation}
Write $K:=[k_{ij}]_{1\le i,j\le n}\in H^\infty_{M_n}$.
Since, by Lemma \ref{lem3.6}, $\mathcal{H}_{\widetilde \theta_2}$ contains an outer function $h$ that is invertible in $H^\infty$, 
it follows from (\ref{3.14-3}) that
$$
k_{ij}(\overline z) h\in H^2\quad\hbox{for each $i,j=1,2,\cdots,n$},
$$
so that $k_{i,j}(\overline z)\in \frac{1}{h}H^2\subseteq H^2$ for each $i,j=1,2,\cdots,n$.
Therefore each $k_{ij}$ is constant and hence, $K$ is constant.
Therefore by (\ref{3.13-2}),
$[T_\Phi^*, T_\Phi]=0$, i.e., $T_\Phi$ is normal.
This completes the proof.
\end{proof}

\bigskip

\begin{rem}\label{rem3.11}
Theorem \ref{thm3.7} may fail if the condition ``$B$ and
$\Theta_2$ are left coprime" is dropped even though $A$ and $\Theta_2$ are left coprime.
To see this, let $\theta$ be a nonconstant
finite Blaschke product. Consider the matrix-valued function
$$
\Phi=\begin{bmatrix} 2\theta+\overline{\theta}&\overline{\theta}\\
\overline{\theta}&2\theta+\overline{\theta}\end{bmatrix}.
$$
Write
$$
\Theta:=\begin{bmatrix}\theta&0\\0&\theta\end{bmatrix}.
$$
Then
$$
\Phi_+=2\Theta \quad \hbox{and} \quad
\Phi_-=\begin{bmatrix}\theta&\theta\\\theta&\theta\end{bmatrix}
=\begin{bmatrix}1&1\\1&1\end{bmatrix}^*\Theta\,.
$$
A direct calculation shows that $\Phi$ is normal. Put
$K:=\frac{1}{2}\begin{bmatrix}1&1\\1&1\end{bmatrix}$.
Remember that for a matrix-valued function $A$, we define
$||A||_\infty:=\sup_{t\in\mathbb T}||A(t)||$
(where $||\cdot||$ means the operator norm).
Then
$||K||_{\infty}=1$ and $\Phi-K\Phi^* \in H^{\infty}_{M_2}$,
so that $T_{\Phi}$ is
hyponormal. Observe that
$$
\aligned ~[T_{\Phi^*},
 T_{\Phi}]&=H_{\Phi_+^*}^*H_{\Phi_+^*}-H_{\Phi_-^*}^*H_{\Phi_-^*}\\
 &=4\begin{bmatrix}H_{\overline{\theta}}^*H_{\overline{\theta}}
 &0\\0&H_{\overline{\theta}}^*H_{\overline{\theta}}\end{bmatrix}-
 2\begin{bmatrix}H_{\overline{\theta}}^*H_{\overline{\theta}}
 &H_{\overline{\theta}}^*H_{\overline{\theta}}\\H_{\overline{\theta}}^*
 H_{\overline{\theta}}&H_{\overline{\theta}}^*H_{\overline{\theta}}\end{bmatrix}\\
&=2\begin{bmatrix}H_{\overline{\theta}}^*H_{\overline{\theta}}
&-H_{\overline{\theta}}^*H_{\overline{\theta}}\\
-H_{\overline{\theta}}^*H_{\overline{\theta}}
&H_{\overline{\theta}}^*H_{\overline{\theta}}\end{bmatrix}\\
&=2\begin{bmatrix}P_{\mathcal H_{\theta}}&-P_{\mathcal H_{\theta}}\\
-P_{\mathcal H_{\theta}}&P_{\mathcal H_{\theta}}\end{bmatrix},\\
\endaligned
$$
which gives
$$
\begin{aligned}
\hbox{ker}[T_{\Phi^*}, T_{\Phi}]=
\hbox{ker}\begin{bmatrix}P_{\mathcal H_{\theta}}&-P_{\mathcal H_{\theta}}\\
-P_{\mathcal H_{\theta}}&P_{\mathcal H_{\theta}}\end{bmatrix}
&=\left\{\begin{bmatrix} f\\ g\end{bmatrix}:\
P_{\mathcal H_{\theta}} f = P_{\mathcal H_{\theta}} g \right\}\\
&=\Theta H^2_{\mathbb C^2} \oplus \bigl\{f\oplus f : f \in \mathcal H_{\theta} \bigr\}.
\end{aligned}
$$
We now claim that $\text{\rm ker}\,[T_{\Phi}^*, T_{\Phi}]$ is invariant under
$T_{\Phi}$.
To show this we suppose
$$
F=\begin{bmatrix}f\\g\end{bmatrix} \in  \text{\rm ker}\,[T_{\Phi}^*,
T_{\Phi}].
$$
Then
$$
T_{\Phi}F= \begin{bmatrix} 2 T_{\theta}
+T_{\overline{\theta}}&T_{\overline{\theta}}\\T_{\overline{\theta}}
&2T_{\theta}+T_{\overline{\theta}}\end{bmatrix}\begin{bmatrix}f\\g\end{bmatrix}
=\begin{bmatrix}2T_{\theta}
f+T_{\overline{\theta}}(f+g)\\2T_{\theta}g+T_{\overline{\theta}}(f+g)\end{bmatrix}.
$$
Observe that
$$
P_{\mathcal H_{\theta}}\bigl(2T_{\theta}
f+T_{\overline{\theta}}(f+g)\bigr)=P_{\mathcal H_{
\theta}}T_{\overline{\theta}}(f+g)
= P_{\mathcal H_{\theta}}\bigl(2T_{\theta}
g+T_{\overline{\theta}}(f+g)\bigr),
$$
which implies that $\text{\rm ker}\,[T_{\Phi}^*, T_{\Phi}]$ is
invariant under $T_{\Phi}$.
But since
$$
\{f\oplus f : f \in \mathcal H_{\theta} \} \subsetneq \mathcal H_{\Theta}\,,
\ \ \hbox{and hence,}\ \ \hbox{ker}\,[T_\Phi^*, T_\Phi]\ne H^2_{\mathbb C^2},
$$
we can see that $T_{\Phi}$ is not normal.
Note that by Lemma \ref{lem3.12},
$\left[\begin{smallmatrix}1&1\\1&1\end{smallmatrix}\right]$ and $\Theta$ are not left
coprime.
\end{rem}

\bigskip

If, in the left coprime factorization $\Phi_-=B^*\Theta_2$
($\Theta_2=\theta_2 I_n$) of Theorem \ref{thm3.7}, $\theta_2$ has a
Blaschke factor, then the assumption of the ``left coprime
factorization" for the analytic part $\Phi_+$ of $\Phi$ can be
dropped in Theorem \ref{thm3.7}.

\medskip

\begin{cor}\label{cor3.8}
Suppose $\Phi=\Phi_-^*+ \Phi_+\in L^\infty_{M_n}$ is such that
$\Phi$ and $\Phi^*$ are of bounded type. In view of (\ref{1.9}), we
may write
$$
\Phi_- = B^*\Theta,
$$
where $\Theta:=\theta I_n$  with an inner function $\theta$. Assume
that $B$ and $\Theta$ are left coprime. Assume also that $\theta$
contains a Blaschke factor. If
\medskip

{\rm (i)} $T_\Phi$ is hyponormal; and

{\rm (ii)} $\text{\rm ker}\,[T_{\Phi}^*, T_{\Phi}]$ is invariant under $T_{\Phi}$
\medskip

\noindent then $T_{\Phi}$ is normal or analytic.
Hence, in particular, if $T_\Phi$ is subnormal then it is normal or analytic.
\end{cor}

\medskip

\begin{proof}
For notational convenience, we let $\Theta_2:=\Theta$ and $\theta_2:=\theta$. Now suppose
$\theta_2$ has a Blaschke factor $b_\alpha$ and write
$B_\alpha:=b_\alpha I_n$. By assumption, $B$ and $B_\alpha$ are left
coprime, so that by Lemma \ref{lem3.12}, $B$ and $B_\alpha$ are
right coprime. Thus, in view of Lemmas \ref{lem3.1} and
\ref{lem3.2}, we can write
$$
\Phi_+ = A^*\Theta_0 \Theta_2=B_\alpha \Delta_1 A_r^* \quad
\hbox{and} \quad \Phi_-=B^*\Theta_2,
$$
where $\Theta_i=\theta_i I_n$ with an inner function $\theta_i$
($i=0,2$), $A_r$ and $B_\alpha \Delta_1$ are right coprime, and $B$
and $\Theta_2$ are left coprime. In particular, we note that $A$ and
$B_\alpha$ are right coprime, so that again by Lemma \ref{lem3.12}, $A$
and $B_\alpha$ are left coprime. On the other hand, an analysis for
the proof of STEP 1 of Theorem \ref{thm3.7} shows that
\begin{equation}\label{3.16-7}
\Theta_0 H^2_{\mathbb C^n}\ \subseteq\ \hbox{\rm ker}\,[T_{\Phi}^*, T_{\Phi}].
\end{equation}
(Note that we didn't employ the assumption ``$A$ and $\Theta_2$ are
left coprime" to get (\ref{3.16-7}) in the proof of STEP 1 of
Theorem \ref{thm3.7}.) Thus if $K\in\mathcal{E}(\Phi)$ then by the
same argument as (\ref{3.15}), we have
\begin{equation}\label{3.16-9}
\hbox{cl ran}\, H_{A\Theta_2^*}\subseteq
\hbox{ker}\,(I-T_{\widetilde{K}}T_{\widetilde{K}}^*).
\end{equation}
Since $A$ and $B_\alpha$ is left coprime it follows that $\widetilde A$ and
$\widetilde B_\alpha$ are right coprime. Thus we can write
$$
\widetilde A\widetilde\Theta_2^*=\widetilde A_1
\widetilde\Omega^*\widetilde B_\alpha^*\ \hbox{(right coprime factorization)}
$$
for some inner function $\Omega$ and $A_1\in H^2_{M_n}$.
It thus follows from Lemma \ref{lem1.1} that
$$
\hbox{cl ran}\, H_{A\Theta_2^*}
    =\Bigl(\hbox{ker}\, H_{\widetilde A\widetilde \Theta_2^*}\Bigr)^{\perp}
          =\Bigl(\widetilde B_\alpha \widetilde\Omega \, H^2_{\mathbb C^n}\Bigr)^\perp
            \supseteq\Bigl(\widetilde B_\alpha H^2_{\mathbb C^n}\Bigr)^\perp
                 =\mathcal{H}_{\widetilde{B}_{\alpha}}.
$$
Thus by (\ref{3.16-9}),
$$
\mathcal{H}_{\widetilde{B}_{\alpha}} \subseteq
\hbox{ker}\,(I-T_{\widetilde{K}}T_{\widetilde{K}}^*).
$$
Then the exactly same argument as the argument from (\ref{3.14-1}) to the end of
the proof of Theorem \ref{thm3.7} with $B_\alpha$ in place of $\Theta_2$
shows that $T_\Phi$ is normal.
(We again note that we didn't employ the assumption ``$A$ and $\Theta_2$ are
left coprime" there.)
This completes the proof.
\end{proof}

\bigskip

We thus have:
\medskip

\begin{cor}\label{cor3.9}
Suppose $\Phi=\Phi_-^*+ \Phi_+\in L^\infty_{M_n}$ is a matrix-valued rational function.
In view of (\ref{1.9}) and (\ref{1.4}), we may write
$$
\Phi_- = B^*\Theta,
$$
where $\Theta:=\theta I_n$  with a finite Blaschke product $\theta$.
Assume that $B(\alpha)$ is invertible for each $\alpha \in \mathcal Z(\theta)$.
If $T_{\Phi}$ is subnormal then $T_{\Phi}$ is normal or analytic.
\end{cor}

\begin{proof} This follows at once from Corollary \ref{cor3.8}
together with Lemma \ref{lem3.12}.
\end{proof}

\vskip 1cm

%%%%%%%%%%%%%%%%%%%%%%%%%%%%%%%%%%%%%%%%%%%%%%%%%%%%%%%%%%%%%%%%%%%%%%%%%%%%%%%%%%%%%%%%%%%%%%%%%
%
%                       SECTION 4
%
%%%%%%%%%%%%%%%%%%%%%%%%%%%%%%%%%%%%%%%%%%%%%%%%%%%%%%%%%%%%%%%%%%%%%%%%%%%%%%%%%%%%%%%%%%%%%%%%%

\section{Scalar Toeplitz operators with finite rank self-commutators}

\noindent
If $\Phi$ is normal and analytic then
$[T_\Phi^*,T_\Phi]=H_{\Phi^*}^*H_{\Phi^*}$, so that by the
Kronecker's Lemma, $T_\Phi$ has a finite rank self-commutator if and
only if $\Phi$ is rational. Therefore Corollary \ref{cor3.9} illustrates the
case of subnormal Toeplitz operators with finite rank
self-commutators.
But it is still open whether
subnormal (even scalar-valued) Toeplitz operators with finite rank
self-commutators are either normal or analytic.
We would like to state:
\medskip

\begin{con}\label{con4.1}
{\it If $T_\phi$ is a subnormal Toeplitz operator with finite rank
self-commutator, then $T_\phi$ is normal or analytic.}
\end{con}

\bigskip

We need not expect that if $T_\phi$ is a
hyponormal Toeplitz operator with finite rank self-commutator then
$\phi$ is of bounded type. Indeed, if $\psi$ is not of bounded type
and $\phi=\overline{\psi}+ z\psi$ (and hence $\phi$ is not of
bounded type) then a straightforward calculation shows that $T_\phi$
is hyponormal and $\hbox{rank}\,[T_\phi^*, T_\phi]=1$.

We would like to take this opportunity to give a positive evidence for
Conjecture \ref{con4.1}. First of all, we recall a theorem of Nakazi and Takahashi
\cite[Theorem 10]{NT} which states that if $T_\phi$ is hyponormal then $[T_\phi^*,
T_\phi]$ is of finite rank if and only if there exists a finite
Blaschke product $b$ in $\mathcal{E}(\phi)$ such that the degree of
$b$ equals the rank of $[T_\phi^*, T_\phi]$.
In what follows we let $b\, \mathcal{M}:=\{
bf: f\in \mathcal M\}$.

\medskip

\begin{thm}\label{thm4.2}
Suppose $T_\phi$
is a hyponormal Toeplitz operator with finite rank self-commutator.
If $\hbox{\rm ker}\,[T_\phi^*, T_\phi]$ and $b\, \hbox{\rm
ker}\,[T_\phi^*, T_\phi]$ (some $b\in\mathcal{E}(\phi)$) are
invariant under $T_\phi$,
then $T_\phi$ is normal or analytic.
\end{thm}

\medskip

\begin{proof}
Write $K:=\text{ker}\,[T_\phi^*, T_\phi]$ and
$R:=\text{ran}\,[T_\phi^*, T_\phi]$.
If
$\phi$ or $\overline{\phi}$ is of bounded type then by
Abrahamse's Theorem, $T_\phi$ is either normal or analytic. Suppose
both $\phi$ and $\overline\phi$ are not of bounded type.
We first claim that
\begin{equation}\label{4.1}
\hbox{\rm cl}\,H_\phi(\hbox{\rm ker}\,[T_\phi^*, T_\phi])=H^2.
\end{equation}
To see this we observe that by the Nakazi-Takahashi Theorem,
there exists a finite Blaschke product $b\in\mathcal{E}(\phi)$ such
that $\text{deg}\,(b)=\text{dim}\, R$.
Since
$$
T_\phi^* T_\phi-T_\phi T_\phi^*
=H_{\overline\phi}^*H_{\overline\phi}-H_{\phi}^*H_{\phi}
=H_{\overline\phi}^*H_{\overline\phi}-H_{b\overline\phi}^*H_{b\overline\phi}
=H_{\overline\phi}^*H_{\overline
b}H_{\overline b}^* H_{\overline\phi},
$$
we have
\begin{equation*}
%%\label{4.2}
\text{ker}\,[T_\phi^*, T_\phi]=\text{ker}\, H_{\overline b}^* H_{\overline\phi}
=\text{ker}\,(T_{\overline\phi}T_b-T_bT_{\overline\phi}),
\end{equation*}
which shows that $H_{\overline{\widetilde
b}}H_{\overline\phi}(K)=0$, and hence $H_{\overline\phi}(K)\subseteq
\widetilde{b}H^2$, so that $\text{cl}\, H_{\overline\phi}K\subseteq
\widetilde{b} H^2$. But since $\text{dim}\,R<\infty$ and by (\ref{1.1}),
$H_{\overline\phi}$ is one-one and has dense range, we have
$$
H^2=\text{cl}H_{\overline\phi}(K+R)=\text{cl}\,(H_{\overline\phi}K+H_{\overline\phi}R)=\text{cl}\,
H_{\overline\phi}K +H_{\overline\phi} R.
$$
We therefore have $\text{cl}\,H_{\overline\phi}K =\widetilde{b}
H^2$ since $\text{dim}\, H_{\overline\phi} R =\text{deg}(b)$. Hence
$$
\text{cl}\,H_\phi K=\text{cl}\, H_{b\overline\phi} K=
\text{cl}T_{\widetilde{b}}^*H_{\overline\phi}K =
T_{\overline{\widetilde{b}}}\widetilde{b}H^2=H^2,
$$
which proves (\ref{4.1}).
On the other hand, we note that $\mathcal{E}(\phi)$ is a singleton set:
otherwise, $\phi$ is of bounded type. Thus
$\mathcal{E}(\phi)$ consists of only a finite Blaschke product $b$.
We next argue that if $T_\phi(bK)\subseteq bH^2$ then
\begin{equation}\label{4.3}
T_\phi(bk)=b\, T_\phi k\quad\text{for each}\ k\in K.
\end{equation}
To see this, let $k\in K$ and write $k_1:=T_\phi k$. Thus $\phi k =
k_1+ \overline{k}_2$ for some $k_2 \in H_0^2=zH^2$. Then
$$
T_\phi(bk)=P(b\phi k)=P(b\overline{k}_2+b k_1)=
P(b\overline{k}_2)+ bk_1=P(b\overline{k}_2)+b\,T_\phi k.
$$
Since, by assumption, $T_\phi(bK)\subseteq bH^2$, it follows
that $P(b\overline{k}_2)\in bH^2$.
But since
$P(b\overline{k}_2)\in (bH^2)^\perp$,
we have
$P(b\overline{k}_2)=0$, which proves (\ref{4.3}). Since $T_\phi T_b-T_b
T_\phi=H_{\overline{\widetilde b}} H_\phi$
it follows from (\ref{4.1}) and (\ref{4.3}) that
$$
H_{\overline{\widetilde b}} H^2=H_{\overline{\widetilde b}}
(\text{cl}\, H_\phi K) =\text{cl}\, H_{\overline{\widetilde b}}
H_\phi K = \text{cl}\,(T_{\phi} T_{b}-T_b T_\phi)K=0,
$$
which implies that $\overline{\widetilde b} H^2\subseteq  H^2$, so
that $b=e^{i\theta}$ for some $\theta\in [0,2\pi)$. Therefore $\phi$
is of the form $\phi=\overline f+e^{i\theta} f$ for some $f\in
H^\infty$ and $\theta\in [0, 2\pi)$ which implies that $T_\phi$ is
normal.
\end{proof}

\bigskip

We  thus have:
\smallskip

\begin{cor}\label{cor4.3}
Suppose $T_\phi$
is a subnormal Toeplitz operator with finite rank self-commutator.
If $b\, \hbox{\rm ker}\,[T_\phi^*, T_\phi]$  is
invariant under $T_\phi$ (some $b\in\mathcal{E}(\phi)$), then $T_\phi$ is normal or analytic.
\end{cor}
\medskip

\begin{proof}
Since $\hbox{ker}\,[T^*,T]$ is invariant under $T$ for
every subnormal operator $T$, the result follows
at once from Theorem \ref{thm4.2}.
\end{proof}
\medskip

We were not unable to decide whether the condition ``$b\, \hbox{\rm
ker}\,[T_\phi^*, T_\phi]$ (some $b\in\mathcal{E}(\phi)$) is
invariant under $T_\phi$" can be dropped from Corollary \ref{cor4.3}:
in other words, if $T_\phi$ is a subnormal operator with finite
rank self-commutator
and $b\in\mathcal{E}(\phi)$,
is $b\, \hbox{\rm ker}\,[T_\phi^*, T_\phi]$ invariant under $T_\phi$\,?
If the answer to this question is affirmative we can conclude that
Conjecture \ref{con4.1} is true.

\bigskip

%%%%%%%%%%%%%%%%%%%%%%%%%%%%%%%%%%%%%%%%%%%%%%%%%%%%%%%%%%%%%%%%%%%%%%%%%%%%%%%%%%%%%%
%
%            Ref
%
%%%%%%%%%%%%%%%%%%%%%%%%%%%%%%%%%%%%%%%%%%%%%%%%%%%%%%%%%%%%%%%%%%%%%%%%%%%%%%%%%%%%

%%%%%%%%%%%%%%%%%%%%%%%%%%%%%%%%%%%%%%%%%%%%%%%%%%%%%%%%%%%%%%%%%%%%%%%%%%%%%%%%%%%%%%%%%%%%%%%%%%%%

\bigskip

\noindent
Department of Mathematics, University of Iowa, Iowa City, IA 52242, U.S.A.\\
E-mail: raul-curto@uiowa.edu

\bigskip

\noindent
Department of Mathematics, Sungkyunkwan University, Suwon 440-746, Korea\\
E-mail: ihwang@skku.edu

\bigskip

\noindent
Department of Mathematics, Seoul National University, Seoul 151-742, Korea\\
E-mail: wylee@snu.ac.kr


\begin{thebibliography}{FKKL}

\bibitem[Ab]{Ab}  M.B. Abrahamse, \textit{Subnormal Toeplitz operators
and functions of bounded type},
Duke Math. J. \textbf{43 }(1976), 597--604.

\bibitem[AIW]{AIW} I. Amemiya, T. Ito, and T.K. Wong,
\textit{On quasinormal Toeplitz operators},
Proc. ~Amer. ~Math. ~Soc. \textbf{50 }(1975), 254--258.

\bibitem[BH]{BH} A. Brown and P.R. Halmos,
\textit{Algebraic properties of Toeplitz operators},
J.~Reine Angew. ~Math. \textbf{213}(1963/1964), 89--102.

\bibitem[Co1]{Co1}  C. Cowen,
\textit{More subnormal Toeplitz operators},
J.~Reine Angew.\ Math. \textbf{367}(1986), 215--219.

\bibitem[Co2]{Co2}  C. Cowen,
\textit{Hyponormal and subnromal Toeplitz operators},
Survey of Soem Recent Results in Operator Theory, I
(J.B. Conway and B.B. Morrel, eds.), Pitman Research Notes
in Mathematics, Volume 171, Longman, 1988, pp. (155-167).

\bibitem[Co3]{Co3}  C. Cowen, \textit{Hyponormality of Toeplitz operators},
Proc. ~Amer. ~Math. ~Soc. \textbf{103}(1988), 809--812.

\bibitem[CoL]{CoL}  C. Cowen and J. Long,
\textit{Some subnormal Toeplitz operators},
J.~Reine Angew. ~Math. \textbf{351}(1984), 216--220.

\bibitem[CuL1]{CuL1}  R.E. Curto and W.Y. Lee,
\textit{Joint hyponormality of Toeplitz pairs}, Memoirs
Amer. ~Math. ~Soc. \textbf{712}, Amer. Math. Soc., Providence, 2001.

\bibitem[CuL2]{CuL2} R.E. Curto and W.Y. Lee,
\textit{Subnormality and $k$-hyponormality of Toeplitz operators:
A brief survey and open question},
Operator Theory and Banach Algebras (Rabat, 1999),
73--81, Theta, Bucharest, 2003.

\bibitem[Do1]{Do1} R.G. Douglas,
\textit{Banach algebra techniques in operator theory}, Academic
Press, New York, 1972.

\bibitem[Do2]{Do2} R.G. Douglas,
\textit{Banach algebra techniques in the theory of Toeplitz
operators}, CBMS 15, Providence, Amer. Math. Soc. 1973.

\bibitem[FL]{FL}  D.R. Farenick and W.Y. Lee,
\textit{Hyponormality and spectra of Toeplitz operators}, Trans.
Amer. Math. Soc. \textbf{348}(1996), 4153--4174.


\bibitem[FF]{FF}  C. Foia\c s and A. Frazho, \textit{The commutant lifting
approach to interpolation problems}, Operator Theory: Adv. Appl.
vol 44, Birkh\" auser, Boston, 1993.

\bibitem[GGK]{GGK} I. Gohberg, S. Goldberg, and M.A. Kaashoek,
\textit{Classes of linear operators, Vol II}, Basel, Birkhauser,
1993.

\bibitem[Gu1]{Gu1} C. Gu,
\textit{A generalization of Cowen's characterization of
hyponormal Toeplitz operators},
 J. Funct. Anal. \textbf{124}(1994), 135--148.

\bibitem[Gu2]{Gu2} C. Gu,
\textit{On a class of jointly hyponormal Toeplitz operators},
 Trans. ~Amer. ~Math. ~Soc.\textbf{354}(2002), 3275--3298.

\bibitem[GHR]{GHR}  C. Gu, J. Hendricks and D. Rutherford,
\textit{Hyponormality of block Toeplitz operators}, Pacific J.
Math. \textbf{223} (2006), 95--111.

\bibitem[GS]{GS} C. Gu and J.E. Shapiro,
\textit{Kernels of Hankel operators and hyponormality of Toeplitz operators},
Math. ~Ann. \textbf{319}(2001), 553--572.

\bibitem[Hal1]{Hal1} P. R. Halmos,
\textit{ Ten problems in Hilbert space},
 Bull. ~Amer. ~Math. ~Soc. \textbf{76}(1970), 887--933.

\bibitem[Hal2]{Hal2} P. R. Halmos,
\textit{Ten years in Hilbert space},
 Integral Equations Operator Theory
 \textbf{2}(1979), 529--564.

\bibitem[HKL1]{HKL1}  I. S. Hwang, I. H. Kim and W.Y. Lee,
\textit{Hyponormality of Toeplitz operators with polynomial symbols},
Math. Ann. \textbf{313(2)} (1999), 247-261.

\bibitem[HKL2]{HKL2}  I. S. Hwang, I. H. Kim and W.Y. Lee,
\textit{Hyponormality of Toeplitz operators with polynomial symbols: An extremal case},
Math. Nach. \textbf{231} (2001), 25-38.

\bibitem[HL1]{HL1}  I. S. Hwang and W.Y. Lee,
\textit{Hyponormality of trigonometric Toeplitz operators},
Trans. Amer. Math. Soc. \textbf{354} (2002), 2461-2474.

\bibitem[HL2]{HL2}  I. S. Hwang and W.Y. Lee,
\textit{Hyponormality of Toeplitz operators with rational
symbols}, Math. Ann. \textbf{335}(2006), 405--414.

\bibitem[HL3]{HL3}  I. S. Hwang and W.Y. Lee,
\textit{Hyponormal Toeplitz operators with rational symbols},
J. Operator Theory \textbf{56}(2006), 47--58.

\bibitem[Le]{Le} W. Y. Lee,
\textit{Cowen sets for Toeplitz operators with finite rank selfcommutators},
J. Operator Theory \textbf{54(2)}(2005), 301-307.

\bibitem[NT]{NT}  T. Nakazi and K. Takahashi,
\textit{Hyponormal Toeplitz operators and extremal problems of
Hardy spaces}, Trans. ~Amer. ~Math. ~Soc. \textbf{338}(1993),
753--769.

\bibitem[Ni]{Ni} N. K. Nikolskii,
\textit{Treatise on the shift operator}, Springer, New York, 1986.

\bibitem[Po]{Po} V.P. Potapov, \textit{On the multiplicative structure of
J-nonexpansive matrix functions}, Tr. Mosk. Mat. Obs. (1955), 125-236 (in Russian);
English trasl. in: Amer. Math. Soc. Transl. {(2) 15}(1966), 131-243.

\bibitem[Ya]{Ya} D.V. Yakubovich, \textit{Real separated algebraic curves, quadrature domains,
Ahlfors type functions and operator theory},
J. ~Funct. ~Anal. \textbf{236}(2006), 25-58.

\bibitem[Zhu]{Zhu} K. Zhu,
\textit{Hyponormal Toeplitz operators with polynomial symbols},
Integral Equations Operator Theory \textbf{21}(1996), 376--381

\end{thebibliography}
\end{document}